\documentclass[12pt]{article}
\usepackage{amsmath,amsfonts,amssymb,amsthm,graphicx,float,color,caption,verbatim}
\usepackage{mathrsfs}\usepackage{hyperref}

\usepackage{geometry}
\usepackage{authblk}
\usepackage[titletoc]{appendix}
\theoremstyle{plain}
\newtheorem{theorem}{Theorem}[section]
\newtheorem{lemma}[theorem]{Lemma}

\theoremstyle{definition}
\newtheorem{definition}[theorem]{Definition}

\newtheorem{remark}[theorem]{Remark}

\theoremstyle{remark}

\title{Initial and Boundary Value Problems for Fractional differential equations involving Atangana-Baleanu  Derivative}
\author [1]{Fatma Al-Musalhi}
\author [1]{Nasser Al-Salti}
\author [2]{Erkinjon Karimov} 
\affil[1] {Department of Mathematics and Statistics, Sultan Qaboos University, P.O. Box 36 Al-Khoudh, Oman}
\affil[2] {Institute of Mathematics named after V.I.Romanovskiy, Academy of Sciences of
the Republic of Uzbekistan, Tashkent 125, Uzbekistan.}

\begin{document}
\maketitle
\begin{abstract}
Initial value problem involving Atangana-Baleanu derivative is considered. An Explicit solution of the given  problem is obtained by reducing the differential equation to Volterra integral equation of second kind and by using  Laplace transform.  To find the solution of the Volterra equation, the successive approximation method is used and a lemma simplifying  the resolvent kernel  has been presented. The use of the given initial value problem is illustrated by considering a boundary value problem in which the solution is expressed in the form of series expansion using orthogonal basis obtained by separation of variables.
\end{abstract}

\section{Introduction and Preliminaries}
\subsection{Introduction and related works}
Recently, two newly definitions of fractional derivative without
singular kernel were suggested, namely, Caputo-Fabrizio fractional derivative \cite{CF} and Atangana-Baleanu fractional derivative \cite{atang}. These new derivatives have been applied to real life problems, for example,  in the fields of  thermal science, material sciences, groundwater modelling and mass-spring system \cite{alkahtani}, \cite{alkatan}, \cite{algahtani}, \cite{alsers}, \cite{atang} and have been considered in a number of other recent work, see for example, \cite{alserk}, \cite{atangk}, \cite{CF2}, \cite{jefainalgahtani}, \cite{LN}, \cite{sheikh}.
The main difference between these two definitions  is that Caputo-Fabrizio  derivative is based on exponential kernel while  Atangana-Baleanu definition used Mittag-leffler function as a non-local kernel.  The non-locality of the kernel gives better description of the memory within structure with different scale.  These two new derivatives  are defined as follows 
\begin{definition}\label{CF}
Let $f\in H^{1}(a,b),$ $b>a,$ $\alpha\in[0,1]$. The Caputo-Fabrizio fractional derivative is defined as 
\begin{equation}
\,^{CF}\,_{a}D^{\alpha}_{t} f(t)=\dfrac{B(\alpha)}{1-\alpha} \int_{a}^{t} f'(s) \exp \left[ \dfrac{-\alpha}{1-\alpha}(t-s)\right]ds,  \end{equation}
and the Atangana-Baleanu fractional  derivative is given by 
\begin{equation}\label{abcd}
\,^{ABC}\,_{a}D^{\alpha}_{t} f(t)=\dfrac{B(\alpha)}{1-\alpha} \int_{a}^{t} f'(s) E_{\alpha} \left[ \dfrac{-\alpha}{1-\alpha}(t-s)^{\alpha}\right]ds,  \end{equation}
where $B(\alpha)$ denotes a normalization function such that $B(0)=B(1)=1$ and $$
E_{\alpha}(z)=\sum_{k=0}^{\infty}\dfrac{z^k}{\Gamma(\alpha k +1)},\quad\text{Re}(\alpha)>0,\,  z\in C,
$$
is the  Mittag-leffler function of one parameter \cite{pod}.
\end{definition} 

For properties related to these derivatives see \cite{atang}, \cite{atangk}, \cite{LN}.  In this paper, we are concerned with solutions to initial and boundary value problems for fractional differential equations involving Atangana-Baleanu derivative. We first recall the Mittag-Leffler of two parameters
$$
E_{\alpha,\beta}(z)=\sum_{k=0}^{\infty}\dfrac{z^k}{\Gamma(\alpha k +\beta)},\quad\text{Re}(\alpha)>0,\text{ Re}(\beta)>0,\,  z\in C,$$
 and a generalized  Mittag-Leffler function 
$$ E^{\delta}_{\alpha,\beta}(z)=\sum_{n=0}^{\infty}\dfrac{(\delta)_n z^n}{\Gamma(\alpha n+\beta) n!},$$
which is introduced by Prabhakar in \cite{Prabhakar},
where $\alpha, \beta, \delta\in\mathbb{C}$ with $\text{Re}(\alpha)>0, $ and $ (\delta)_n=\dfrac{\Gamma(\delta+n)}{\Gamma(\delta)}$ is Pochammer's symbol. For $\delta=1,$ it is reduced to Mittag-Leffler function. \\ Moreover, Mittag-Leffler function $E_{\alpha}(\lambda t^\alpha)$ is bounded(see \cite{Prabhakar}), i.e, 
\begin{equation}\label{bm}
E_{\alpha}(\lambda t^\alpha)\leq M,
\end{equation}
where $M$ denotes a positive constant.
 In \cite{atang}, Atangana and Baleanu considered the time fractional ordinary differential equation 
 $$ \,^{ABC}\,_{0}D^{\alpha}_{t} f(t)=u(t),$$
and on using Laplace transform they found the following solution 
$$f(t)= \,^{AB}\,_{0}I^{\alpha}_{t}u(t)=\dfrac{1-\alpha}{B(\alpha)} u(t)+\dfrac{\alpha}{B(\alpha)\Gamma(\alpha)}\int_{0}^{t}u(s)(t-s)^{\alpha-1} ds,$$
where they have defined
$$ \,^{AB}\,_{a}I^{\alpha}_{t}f(t)=\dfrac{1-\alpha}{B(\alpha)} f(t)+\dfrac{\alpha}{B(\alpha)\Gamma(\alpha)}\int_{a}^{t}f(s)(t-s)^{\alpha-1} ds$$
to be the fractional integral associated with  the fractional derivative $(\ref{abcd}).$ 
In this paper, we consider the following  initial value problem (IVP)
\begin{equation}\label{ivp}
\begin{array}{c}
\,^{ABC}\,_{0}D^{\alpha}_{t} u(t)-\lambda u(t)=f(t),\quad t\geq 0,\\
 u(0)=u_{0},
\end{array}
\end{equation}
where $\lambda,\,u_{0}\in \mathbb{R}.$  The solution of this IVP is obtained by two different methods, namely, by reducing it to Volterra integral equation of the second kind and by using Laplace transform. The use of such IVP is illustrated by considering a boundary value problem in which the solution is expressed in the form of series expansion using orthogonal basis obtained by separation of variables. The rest of the paper is organized as follows: at the end of this section, we present a lemma which is important for simplifying the resolvent kernel of the Volterra equation. Then, section $2$ is devoted for our main result which is the explicit solution of the  IVP $(\ref{ivp})$.  We conclude this paper by considering a boundary value problem where we have utlized the solution of the IVP $(\ref{ivp})$.
\subsection{Preliminaries}
As metioned earlier, one way to solve the IVP $(\ref{ivp})$ is to reduce it to a  Volterra integral equation and in order to simplify our calculations, namely the resolvent kernel of the Volterra equation, we have established the following Lemma:
\begin{lemma}\label{sum}
Let $\lambda\in\mathbb{R}$ and $0<\alpha<1$ with $\lambda\neq \dfrac{B(\alpha)}{1-\alpha}$, then  
\begin{equation}
\begin{array}{c}
\displaystyle\sum_{i=1}^{\infty}\left(\dfrac{B(\alpha)}{B(\alpha)-\lambda(1-\alpha)} \right)^i\left(\dfrac{\alpha}{1-\alpha} \right)^i(t-\xi)^{i\alpha-1}E^i_{\alpha , i\alpha}\left( \dfrac{-\alpha}{1-\alpha}(t-\xi)^\alpha\right)\\=\dfrac{B(\alpha)}{B(\alpha)-\lambda (1-\alpha)}\dfrac{\alpha}{1-\alpha}(t-\xi)^{\alpha-1}E_{\alpha,\alpha}\left[ \dfrac{\alpha\lambda}{B(\alpha)-\lambda(1-\alpha)}(t-\xi)^{\alpha}\right].\end{array}
\end{equation}
\end{lemma}
\begin{proof}

We begin by expanding the following series:
$$\begin{array}{l}
\displaystyle\sum_{i=1}^{\infty}\left(\dfrac{B(\alpha)}{B(\alpha)-\lambda(1-\alpha)} \right)^i\left(\dfrac{\alpha}{1-\alpha} \right)^i(t-\xi)^{i\alpha-1}E^i_{\alpha , i\alpha}\left( \dfrac{-\alpha}{1-\alpha}(t-\xi)^\alpha\right)\\

=\left(\dfrac{B(\alpha)}{B(\alpha)-\lambda(1-\alpha)} \right)\left(\dfrac{\alpha}{1-\alpha} \right) (t-\xi)^{\alpha-1}E_{\alpha , \alpha}\left( \dfrac{-\alpha}{1-\alpha}(t-\xi)^\alpha\right)+\\
\quad\left(\dfrac{B(\alpha)}{B(\alpha)-\lambda(1-\alpha)} \right)^2\left(\dfrac{\alpha}{1-\alpha} \right)^2 (t-\xi)^{2\alpha-1}E^2_{\alpha , 2\alpha}\left( \dfrac{-\alpha}{1-\alpha}(t-\xi)^{2\alpha}\right)+\\
\quad\left(\dfrac{B(\alpha)}{B(\alpha)-\lambda(1-\alpha)} \right)^3\left(\dfrac{\alpha}{1-\alpha} \right)^3 (t-\xi)^{3\alpha-1}E^3_{\alpha , 3\alpha}\left( \dfrac{-\alpha}{1-\alpha}(t-\xi)^{3\alpha}\right)+
\cdots.\\
\end{array}
$$
Using the definition of  Mittag-Leffler functions, it can be written as follows:
$$\begin{array}{l}
\left(\dfrac{B(\alpha)}{B(\alpha)-\lambda(1-\alpha)} \right)\left(\dfrac{\alpha}{1-\alpha} \right)(t-\xi)^{\alpha-1}\displaystyle\sum_{n=0}^{\infty}\dfrac{(1)_n\left( \dfrac{-\alpha}{1-\alpha}(t-\xi)^{\alpha}\right)^n}{\Gamma (\alpha n+\alpha)n!} +\\
\left(\dfrac{B(\alpha)}{B(\alpha)-\lambda(1-\alpha)} \right)^2\left(\dfrac{\alpha}{1-\alpha} \right)^2(t-\xi)^{2\alpha-1}\displaystyle\sum_{n=0}^{\infty}\dfrac{(2)_n\left( \dfrac{-\alpha}{1-\alpha}(t-\xi)^{\alpha}\right)^n}{\Gamma (\alpha n+2\alpha)n!}+\\
\left(\dfrac{B(\alpha)}{B(\alpha)-\lambda(1-\alpha)} \right)^3\left(\dfrac{\alpha}{1-\alpha} \right)^3(t-\xi)^{3\alpha-1}\displaystyle\sum_{n=0}^{\infty}\dfrac{(3)_n\left( \dfrac{-\alpha}{1-\alpha}(t-\xi)^{\alpha}\right)^n}{\Gamma (\alpha n+3\alpha)n!}+\cdots ,\\

\end{array}
$$
and expanding the series representations of Mittag-Leffler functions, gives
$$\begin{array}{l}
\left(\dfrac{B(\alpha)}{B(\alpha)-\lambda(1-\alpha)} \right)\left(\dfrac{\alpha}{1-\alpha}\right)  (t-\xi)^{\alpha-1}\left[ \dfrac{1}{\Gamma(\alpha)}-\left( \dfrac{\alpha}{1-\alpha}\right)\dfrac{ (t-\xi)^\alpha}{\Gamma(2\alpha)}
+ \left( \dfrac{\alpha}{1-\alpha}\right) ^2\dfrac{(t-\xi)^{2\alpha}}{\Gamma(3\alpha)}+\cdots \right] \\
+\left(\dfrac{B(\alpha)}{B(\alpha)-\lambda(1-\alpha)} \right)^{2}\left(\dfrac{\alpha}{1-\alpha}\right)^{2}  (t-\xi)^{2\alpha-1}\left[ \dfrac{1}{\Gamma(2\alpha)}-2\left( \dfrac{\alpha}{1-\alpha}\right)\dfrac{ (t-\xi)^\alpha}{\Gamma(3\alpha)}
+ \cdots \right] \\
+\left(\dfrac{B(\alpha)}{B(\alpha)-\lambda(1-\alpha)} \right)^{3}\left(\dfrac{\alpha}{1-\alpha}\right)^{3}  (t-\xi)^{3\alpha-1}\left[ \dfrac{1}{\Gamma(3\alpha)}
+ \cdots \right].
\end{array}
$$
Now, combining like terms, we have
$$\begin{array}{l}
\left(\dfrac{B(\alpha)}{B(\alpha)-\lambda(1-\alpha)} \right)\left(\dfrac{\alpha}{1-\alpha}\right)\dfrac{(t-\xi)^{\alpha-1}}{\Gamma(\alpha)}+
\left(\dfrac{B(\alpha)}{B(\alpha)-\lambda(1-\alpha)} \right)\left(\dfrac{\alpha}{1-\alpha}\right)^2\dfrac{(t-\xi)^{2\alpha-1}}{\Gamma(2\alpha)} \\
\left[\dfrac{B(\alpha)}{B(\alpha)-\lambda(1-\alpha)} -1 \right]+\left(\dfrac{B(\alpha)}{B(\alpha)-\lambda(1-\alpha)} \right)\left(\dfrac{\alpha}{1-\alpha}\right)^3\dfrac{(t-\xi)^{3\alpha-1}}{\Gamma(3\alpha)} \left[1-2\dfrac{B(\alpha)}{B(\alpha)-\lambda(1-\alpha)} \right. \\\left. +\left( \dfrac{B(\alpha)}{B(\alpha)-\lambda(1-\alpha)}\right) ^2 \right]+\cdots,\\
\end{array}
 $$  
and simplifying further gives
$$\begin{array}{l}
\left(\dfrac{B(\alpha)}{B(\alpha)-\lambda(1-\alpha)} \right)\left(\dfrac{\alpha}{1-\alpha}\right)(t-\xi)^{\alpha-1}\\\left[\dfrac{1} {\Gamma(\alpha)}+\dfrac{\alpha\lambda}{{B(\alpha)-\lambda(1-\alpha)} }\dfrac{(t-\xi)^{\alpha}}{\Gamma(2\alpha)}+ \left( \dfrac{\alpha\lambda}{{B(\alpha)-\lambda(1-\alpha)} }\right)^2\dfrac{(t-\xi)^{2\alpha}}{\Gamma(3\alpha)}+\cdots\right] \\
=\left(\dfrac{B(\alpha)}{B(\alpha)-\lambda(1-\alpha)} \right)\left(\dfrac{\alpha}{1-\alpha}\right)(t-\xi)^{\alpha-1}E_{\alpha,\alpha}\left[ \dfrac{\alpha\lambda}{B(\alpha)-\lambda(1-\alpha)}t^{\alpha}\right].
\end{array}$$

\end{proof}

\section{Main Result}
\subsection{Initial value problem}
Here, we consider the following problem:\\
Find a solution $u(t)\in H^{1}(0,T)$ that satisfies the following  equation
\begin{equation}\label{eq}
\,^{ABC}\,_{0}D^{\alpha}_{t} u(t)-\lambda u(t)=f(t),\qquad 0\leq t\leq T,\end{equation}
and the initial condition
\begin{equation}\label{ic}
u(0)=u_0,
\end{equation}
where $\lambda, u_0 \in\mathbb{R}$. The solution of this initial value problem is formulated in the following theorem:
\begin{theorem} If $\lambda\neq \dfrac{B(\alpha)}{1-\alpha}$, $f(t)\in C(0,T)$ and $f(0)=-\lambda u_0$, then the solution of the initial value problem $(\ref{eq})-(\ref{ic})$ is given by 
\begin{equation}\label{soln}
\begin{array}{ll}
u(t)&=\dfrac{B(\alpha)u_0}{B(\alpha)-\lambda(1-\alpha)}  E_{\alpha} \left[ \dfrac{\alpha \lambda }{B(\alpha)-\lambda(1-\alpha)}t^{\alpha}\right]+\dfrac{1-\alpha}{B(\alpha)-\lambda(1-\alpha)}f(t) \\&+ \dfrac{\alpha B(\alpha)}{\left(B(\alpha)-\lambda(1-\alpha)\right)^2}\displaystyle\int_{0}^{t} f(\xi) (t-\xi)^{\alpha-1}E_{\alpha,\alpha}\left[ \dfrac{\alpha\lambda}{B(\alpha)-\lambda(1-\alpha)}(t-\xi)^{\alpha}\right] d\xi.
\end{array} \end{equation}

\end{theorem}
\begin{proof}
Using the definition of Atangana-Baleanu fractional derivative, we have 
 $$\dfrac{B(\alpha)}{1-\alpha} \int_{0}^{t} u'(s) E_{\alpha} \left[ \dfrac{-\alpha}{1-\alpha}(t-s)^{\alpha}\right]ds-\lambda u(t) =f(t),$$
which on integrating by parts leads to 
$$
\begin{array}{c}
\left( \dfrac{B(\alpha)}{1-\alpha}-\lambda \right) u(t)-\dfrac{B(\alpha)}{1-\alpha}\displaystyle\int_{0}^{t} 
\dfrac{d}{ds }\left(  E_{\alpha} \left[ \dfrac{-\alpha}{1-\alpha}(t-s)^{\alpha}\right]\right)  u(s)ds =\\
f(t)+\dfrac{B(\alpha) u_0}{1-\alpha}  E_{\alpha} \left[ \dfrac{-\alpha}{1-\alpha}\,t^{\alpha}\right].
\end{array}$$
For $\lambda\neq \dfrac{B(\alpha)}{1-\alpha}$ , one can write the above equation as a Volterra integral equation of the second kind
\begin{equation}\label{inteq}
u(t)-\int_{0}^{t}u(s)K(t,s)ds=\widehat{f}(t),
\end{equation}
where 
$$K(t,s)= \dfrac{B(\alpha)}{B(\alpha)-\lambda(1-\alpha)}\dfrac{d}{ds } E_{\alpha} \left[ \dfrac{-\alpha}{1-\alpha}(t-s)^{\alpha}\right], $$
and
$$\widehat{f}(t)=\dfrac{B(\alpha)u_0}{B(\alpha)-\lambda(1-\alpha)} E_{\alpha} \left[ \dfrac{-\alpha}{1-\alpha}t^{\alpha}\right]+\dfrac{1-\alpha}{B(\alpha)-\lambda(1-\alpha)}f(t). $$
To solve equation $(\ref{inteq}),$ we use successive approximation method starting with 
$u_0(t)=\widehat{f}(t).$
Then,
$$\begin{array}{ll}
u_1(t)&=\widehat{f}(t)+\displaystyle\int_{0}^{t}u_{0}(s)\,K(t,s)ds\\
&=\widehat{f}(t)+\displaystyle\int_{0}^{t}\widehat{f}(s)\,K(t,s)ds,\end{array}$$
and similarly we obtain $u_2(t)$
$$\begin{array}{ll}
u_2(t)&=\widehat{f}(t)+\displaystyle\int_{0}^{t}u_{1}(s)\,K(t,s)ds\\
&=\widehat{f}(t)+\displaystyle\int_{0}^{t}\left( \widehat{f}(s)+\int_{0}^{s}\widehat{f}(\xi)K(s,\xi)d\xi\right) \,K(t,s)\,ds\\
&=\widehat{f}(t)+\displaystyle\int_{0}^{t} \widehat{f}(s)K(t,s)ds +\int_{0}^{t}\widehat{f}(\xi)d\xi \int_{\xi}^{t} K(s,\xi)\,K(t,s)\,ds\\
\end{array}$$
Set $K_2(t,\xi)=\displaystyle\int_{\xi}^{t} K(t,s)\,K(s,\xi) ds,$
so that 
 $$
 \begin{array}{ll}
 u_{2}(t)&= \widehat{f}(t)+\displaystyle\int_{0}^{t} \widehat{f}(s)K(t,s)ds +\int_{0}^{t}\widehat{f}(\xi)K_{2}(t,\xi)d\xi\\
 &=  \widehat{f}(t)+\displaystyle\int_{0}^{t} \widehat{f}(\xi)\left[ K_1(t,\xi) +K_{2}(t,\xi)\right] d\xi.\\
  \end{array} $$
Continuing the same process, the $n^{\text{th}}$ term will have the following form  
 $$u_{n}(t)= \widehat{f}(t)+\int_{0}^{t} \widehat{f}(\xi)\sum_{i=1}^{n}K_{i}(t,\xi)d\xi,$$
where 
 $$K_{1}(t,\xi)=K(t,\xi),\quad K_{i}(t,\xi)=\int_{\xi}^{t}K(t,s)K_{i-1}(s,\xi) ds,\quad i=2,3,\cdots,$$
which can be derived using mathematical induction.\\
 To obtain the general expression for the kernel $K_{i}(t,\xi)$, we substitute for $K(t,s)$  and start with
  $$\begin{array}{ll}
 K_{2}(t,\xi)&=\displaystyle\int_{\xi}^{t} K(t,s)K_{1}(s,\xi)ds\\
 &=\left( \dfrac{B(\alpha)}{B(\alpha)-\lambda(1-\alpha)}\right)^{2}\displaystyle\int_{\xi}^{t}\dfrac{d}{ds}\left( E_{\alpha} \left[ \dfrac{-\alpha}{1-\alpha}(t-s)^{\alpha}\right]\right) \dfrac{d}{d\xi}\left( E_{\alpha} \left[ \dfrac{-\alpha}{1-\alpha}(s-\xi)^{\alpha}\right]\right)  ds \\
&= \displaystyle\left( \dfrac{B(\alpha)}{B(\alpha)-\lambda(1-\alpha)}\right)^{2}\left( \dfrac{-\alpha}{1-\alpha}\right)^{2}\int_{\xi}^{t} (t-s)^{\alpha-1} E_{\alpha,\alpha} \left[ \dfrac{-\alpha}{1-\alpha}(t-s)^{\alpha}\right]\\
&\quad \qquad\qquad\qquad\qquad\qquad\qquad\qquad\qquad(s-\xi)^{\alpha-1} E_{\alpha,\alpha} \left[ \dfrac{-\alpha}{1-\alpha}(s-\xi)^{\alpha}\right]ds,\\
 \end{array}$$
whereupon using Theorem $5.$ in \cite{Prabhakar},  $K_{2}(t,\xi)$ reduces to  
  $$K_{2}(t,\xi)=\displaystyle\left( \dfrac{B(\alpha)}{B(\alpha)-\lambda(1-\alpha)}\right)^{2}\left( \dfrac{\alpha}{1-\alpha}\right)^{2} (t-\xi)^{2\alpha-1} E^{2}_{\alpha,2\alpha}\left[ \dfrac{-\alpha}{1-\alpha}(t-\xi)^{\alpha}\right].     $$
Repeating the same procedure for $K_{3}(t,\xi)$, we have
  $$\begin{array}{ll}
  K_{3}(t,\xi)&=\displaystyle\int_{\xi}^{t} K(t,s)K_{2}(s,\xi)ds\\
  &=\displaystyle\left( \dfrac{B(\alpha)}{B(\alpha)-\lambda(1-\alpha)}\right)^{3}\left( \dfrac{\alpha}{1-\alpha}\right)^{2}\displaystyle \int_{\xi}^{t}\dfrac{d}{ds}\left( E_{\alpha} \left[ \dfrac{-\alpha}{1-\alpha}(t-s)^{\alpha}\right]\right)\\
  &\quad\qquad \qquad\qquad\qquad\qquad\qquad\qquad\qquad(s-\xi)^{2\alpha-1} E^{2}_{\alpha,2\alpha}\left[ \dfrac{-\alpha}{1-\alpha}(s-\xi)^{\alpha}\right] ds \\
  &=\displaystyle\left( \dfrac{B(\alpha)}{B(\alpha)-\lambda(1-\alpha)}\right)^{3}\left( \dfrac{\alpha}{1-\alpha}\right)^{3}\displaystyle \int_{\xi}^{t}(t-s)^{\alpha-1}E_{\alpha,\alpha} \left[ \dfrac{-\alpha}{1-\alpha}(t-s)^{\alpha}\right]\\
  &\qquad\qquad\qquad\qquad\qquad\qquad\qquad\qquad\quad (s-\xi)^{2\alpha-1} E^{2}_{\alpha,2\alpha}\left[ \dfrac{-\alpha}{1-\alpha}(s-\xi)^{\alpha}\right] ds, \\
 \end{array}$$
 and end with 
 $$ K_{3}(t,\xi)=\displaystyle\left( \dfrac{B(\alpha)}{B(\alpha)-\lambda(1-\alpha)}\right)^{3}\left( \dfrac{\alpha}{1-\alpha}\right)^{3} (t-\xi)^{3\alpha-1} E^{3}_{\alpha,3\alpha}\left[ \dfrac{-\alpha}{1-\alpha}(t-\xi)^{\alpha}\right].$$
Consequently, the general expression for the kernel is given by
  $$K_{i}(t,\xi)=  \left( \dfrac{B(\alpha)}{B(\alpha)-\lambda(1-\alpha)}\right)^{i}\left( \dfrac{\alpha}{1-\alpha}\right)^{i} (t-\xi)^{i\alpha-1} E^{i}_{\alpha,i\alpha}\left[ \dfrac{-\alpha}{1-\alpha}(t-\xi)^{\alpha}\right],\, i=1,2,3,\cdots. $$
As  $n\rightarrow \infty$, the approximations of $u_{n}(t)$ converges  to the solution $u(t)$
$$u(t)=   \widehat{f}(t)+\int_{0}^{t}  \widehat{f}(\xi) \sum_{i=1}^{\infty} K_{i}(t,\xi)d\xi.$$
According to Lemma \ref{sum}, we get the following 
$$u(t)=\widehat{f}(t)+ \left( \dfrac{B(\alpha)}{B(\alpha)-\lambda(1-\alpha)}\right)\left( \dfrac{\alpha}{1-\alpha}\right)\int_{0}^{t} \widehat{f}(\xi) (t-\xi)^{\alpha-1}E_{\alpha,\alpha}\left[ \dfrac{\alpha\lambda}{B(\alpha)-\lambda(1-\alpha)}(t-\xi)^{\alpha}\right] d\xi,
 $$
simplifying the above integral using formula ($1.107$) in \cite{pod} and properties of Mittag-Leffler function, we obtain the desired solution given by $(\ref{soln})$. \\
An alternative way of finding the solution of the initial value problem $(\ref{eq})-(\ref{ic})$  is using Laplace transform method.
So, by applying Laplace transform to both sides of equation $(\ref{eq})$, we have 
$$\dfrac{B(\alpha)}{1-\alpha} \dfrac{s^{\alpha}U(s)-s^{\alpha-1}u(0)}{s^{\alpha}+\dfrac{\alpha}{1-\alpha}}-\lambda U(s)=F(s),$$
where $U(s)=\mathcal{L}\{u(t)\}(s)$ and $$\mathcal{L}\{ \,^{ABC}\,_{0}D^{\alpha}_{t} u(t)\}(s)=\dfrac{B(\alpha)}{1-\alpha} \dfrac{s^{\alpha}U(s)-s^{\alpha-1}u(0)}{s^{\alpha}+\dfrac{\alpha}{1-\alpha}}.$$
Simplifying and solving for $U(s)$, we get

$$\begin{array}{l}
U(s)= \dfrac{B(\alpha) s^{\alpha-1}u_0}{s^{\alpha} (B(\alpha)-\lambda(1-\alpha))-\lambda\alpha}+\dfrac{(1-\alpha)s^\alpha +\alpha}{s^{\alpha} (B(\alpha)-\lambda(1-\alpha))-\lambda\alpha}F(s),
\end{array}$$
which can be rewritten as
$$\begin{array}{ll}
U(s)&= \dfrac{B(\alpha) s^{\alpha-1}u_0}{\left( B(\alpha)-\lambda(1-\alpha) \right)\left[  s^{\alpha} -\dfrac{\lambda\alpha}{B(\alpha)-\lambda(1-\alpha)}\right] }\\&+\dfrac{(1-\alpha)s^\alpha+\alpha }{\left( B(\alpha)-\lambda(1-\alpha) \right)\left[  s^{\alpha} -\dfrac{\lambda\alpha}{B(\alpha)-\lambda(1-\alpha)}\right] }F(s).
\end{array}$$
Since the Laplace transform of Mittag-Leffler function is given by  $$ \mathcal{L}\{t^{\beta-1}E_{\alpha,\beta}(\lambda t^{\alpha})\}(s)=\dfrac{s^{\alpha-\beta}}{s^\alpha -\lambda},$$
then, applying Laplace inverse gives
$$\begin{array}{ll}
u(t)&=\dfrac{B(\alpha)u_0 }{ B(\alpha)-\lambda(1-\alpha) } E_{\alpha,1}\left( \dfrac{\alpha\lambda}{B(\alpha)-\lambda(1-\alpha)} t^{\alpha}\right) \\&+ \dfrac{1-\alpha}{B(\alpha)-\lambda(1-\alpha)}\left(\dfrac{d}{dt} E_{\alpha,1}\left( \dfrac{\alpha\lambda}{B(\alpha)-\lambda(1-\alpha)} t^{\alpha}\right) *f(t)\right)\\&+\dfrac{(1-\alpha)f(t)}{B(\alpha)-\lambda(1-\alpha)} +\dfrac{\alpha}{B(\alpha)-\lambda(1-\alpha)}\left( t^{\alpha-1} E_{\alpha,\alpha}\left( \dfrac{\alpha\lambda}{B(\alpha)-\lambda(1-\alpha)} t^{\alpha}\right)*f(t)\right).
\end{array}$$
Consequently, 
$$\begin{array}{ll}
u(t)&=\dfrac{B(\alpha)u_0}{ B(\alpha)-\lambda(1-\alpha) } E_{\alpha,1}\left( \dfrac{\alpha\lambda}{B(\alpha)-\lambda(1-\alpha)} t^{\alpha}\right)+\dfrac{(1-\alpha)}{B(\alpha)-\lambda(1-\alpha)} f(t)\\&
+ \dfrac{\alpha B(\alpha)}{B(\alpha)-\lambda(1-\alpha))^2}\left( t^{\alpha-1} E_{\alpha,\alpha}\left( \dfrac{\alpha\lambda}{B(\alpha)-\lambda(1-\alpha)} t^{\alpha}\right)*f(t)\right),
\end{array}$$
which is the same as the solution obtained by successive iterations.
  \end{proof}
\begin{remark} For the case $\lambda=0 $ and $u(0)=0$, we get 
  $$u(t)=\dfrac{1-\alpha}{B(\alpha)}f(t)+\dfrac{\alpha}{B(\alpha)\Gamma(\alpha)}\int_{0}^{t}f(\xi) (t-\xi)^{\alpha-1} d\xi,$$ 
  which coincides with the result obtained in \cite{atang}.
  \end{remark}

\subsection{Boundary value problem}
Now, we consider a direct problem of determining $ u(x,t) $ in a  rectangular domain\\ $\Omega=\left\lbrace (x,t):0<x<1,\; 0<t<T\right\rbrace,$ such that $u \in C^{2}(0, 1) \times H^{1}(0, T)$ and satisfies the following  initial-boundary value problem:
\begin{eqnarray}
&&\,^{ABC}\,_{0}D^{\alpha}_{t} u(x,t)-u_{xx}(x,t)=f(x,t), \quad\quad \quad (x,t) \in \Omega \label{sinverse}\\
&&\quad \quad u(0,t)=0,\quad u(1,t)=0,\hspace{2.2cm} 0\leq t\leq T,\label{ISBC}\\
&&\quad\quad\, u(x,0)=0,\;\quad\quad\quad \quad\quad\quad \quad \hspace{1.5cm} 0\leq x\leq 1, \label{ISIC}
\end{eqnarray}
where $f(x,t)$ is a given function.
We begin by using separation of variables method to solve the homogeneous equation corresponding to equation (\ref{sinverse}) aloong with the boundary conditions  (\ref{ISBC}). Thus, we obtain the following spectral problem:
\begin{eqnarray}\label{sepr}
 \left\lbrace \begin{array}{l}
X''+\lambda X=0,\\
X(0)=0,\quad X(1)=0.
\end{array}\right.
\end{eqnarray}
which is self adjoint and has the following eigenvalues  $$\lambda_k=(k\pi)^2,\quad k=1,2,3,\cdots .$$
The corresponding eigenfunctions are
\begin{equation} \label{sys1}
X_{k}=\sin(k \pi x) \quad k=1,2,3,\cdots.
\end{equation}
Since the system of eigenfunctions (\ref{sys1}) forms an orthogonal basis in $L^{2}(0,1)$, we can  then write the solution $u(x,t)$  and the given function $f(x,t)$  in the form of series expansions as follows:
\begin{equation}\label{solu}
u(x,t)=\sum_{k=1}^{\infty} u_k(t) \sin(k \pi x), 
\end{equation}
\begin{equation}\label{souf}
f(x,t)=\sum_{k=1}^{\infty} f_k(t) \sin(k \pi x), 
\end{equation}
where $u_{k}(t)$ is the unknown to be found and $f_{k}(t)$ is given by 
$f_{k}(t)=2\int_{0}^{1}f(x,t)\sin(k\pi x)dx.$
Substituting $(\ref{solu})$ and $(\ref{souf})$ into $(\ref{sinverse})$ and $(\ref{ISIC})$, we obtain the following  fractional differential equation
\begin{equation}\label{fde}
 \,^{ABC}\,_{0}D^{\alpha}_{t} u_k(t)+ k^2 \pi^2 u_k(t)=f_{k}(t),\end{equation}
along with the following condition
\begin{equation}\label{fdecond}
u_k(0)=0.
\end{equation}
Whereupon using Theorem $2.1$, the solution is given by 
$$\begin{array}{ll}
u_k(t)&= \dfrac{1-\alpha}{B(\alpha)+k^2\pi^2(1-\alpha)}f_k(t)+\dfrac{\alpha\, B(\alpha)}{(B(\alpha)+k^2\pi^2(1-\alpha))^2}\\& \displaystyle\int_{0}^{t} f_k(\xi) (t-\xi)^{\alpha-1}E_{\alpha,\alpha}\left[ \dfrac{-\alpha k^2\pi^2}{B(\alpha)+k^2\pi^2(1-\alpha)}(t-\xi)^{\alpha}\right]d\xi,
\end{array} $$
with $f_{k}(0)=0$, which is acheived by assuming  $f(x,0)=0$.
Thus,   the solution  $u(x,t)$ can now be written as $$\begin{array}{ll}
u(x,t)&= 
\displaystyle\sum_{k=1}^{\infty}\left( \dfrac{1-\alpha}{B(\alpha)+k^2\pi^2(1-\alpha)}f_k(t)+\dfrac{\alpha\, B(\alpha)}{(B(\alpha)+k^2\pi^2(1-\alpha))^2}\right. \\&\left.  \displaystyle\int_{0}^{t} f_k(\xi) (t-\xi)^{\alpha-1}E_{\alpha,\alpha}\left[ \dfrac{-\alpha k^2\pi^2}{B(\alpha)+k^2\pi^2(1-\alpha)}(t-\xi)^{\alpha}\right]d\xi\right) \sin(k \pi x).\end{array}$$
In order to complete the proof of existence, we  need to show the uniform convergence of the  series representations of $$u(x,t), \,u_{x}(x,t),\,u_{xx}(x,t),\,^{ABC}\,_{0}D^{\alpha}_{t} u(x,t).$$
Since Mittag-Leffler functions is bounded, then  the  uniform convergence of the series representation of $u(x,t)$ is ensured by assuming $f(x, \cdot) \in C(0, T)$.
Now, the series representation of $u_{xx}(x,t)$ is given by 
$$\begin{array}{ll}
u_{xx}(x,t)&=
-\displaystyle\sum_{k=1}^{\infty}\left( \dfrac{k^2\pi^2(1-\alpha)}{B(\alpha)+k^2\pi^2(1-\alpha)}f_k(t)+\dfrac{k^2\pi^2\alpha\, B(\alpha)}{(B(\alpha)+k^2\pi^2(1-\alpha))^2}\right. \\&\left.  \displaystyle\int_{0}^{t} f_k(\xi) (t-\xi)^{\alpha-1}E_{\alpha,\alpha}\left[ \dfrac{-\alpha k^2\pi^2}{B(\alpha)+k^2\pi^2(1-\alpha)}(t-\xi)^{\alpha}\right]d\xi\right) \sin(k \pi x)\\
&= \displaystyle\sum_{k=1}^{\infty}f_{k}(t)\sin(k\pi x)+\displaystyle\sum_{k=1}^{\infty}\dfrac{B(\alpha)}{B(\alpha)+k^2\pi^2(1-\alpha)}\sin(k\pi x)\\& \displaystyle\int_{0}^{t} f'_k(\xi)E_{\alpha,1}\left[ \dfrac{-\alpha k^2\pi^2}{B(\alpha)+k^2\pi^2(1-\alpha)}(t-\xi)^{\alpha}\right]d\xi. 

\end{array}$$
Assuming $f_{t}(x,t)$ is integrable, it is clear that the second term of the above series converges uniformly. For convergence of the first term, we assume $f(0,t)=f(1,t)=0$ and use integration by parts  to get 
$$\begin{array}{ll}
\bigg|\displaystyle\sum_{k=1}^{\infty}f_{k}(t)\sin(k\pi x)\bigg|=\bigg|\displaystyle\sum_{k=1}^{\infty}\frac{1}{k \pi}f_{1k}(t)\sin(k\pi x)\bigg|

\leq\displaystyle\sum_{k=1}^{\infty} \dfrac{ 1}{k\pi} |f_{1k}(t)|, \end{array}$$
 where
$$
f_{1k}(t)=2\displaystyle\int_{0}^{1}f_{x}(x,t) \cos(k\pi x) \,dx.$$
Using the inequality $ab \leq \dfrac{1}{2}(a^2+b^2)$ and the  Bessel's  inequality, we then have the following estimate
$$\begin{array}{ll}
\bigg|\displaystyle\sum_{k=1}^{\infty}f_{k}(t)\sin(k\pi x)\bigg|&\leq 
\displaystyle\sum_{k=1}^{\infty} \dfrac{1}{2}\left( \dfrac{ 1}{k^2\pi^2} +|f_{1k}|^2\right) \\
&\leq \displaystyle\sum_{k=1}^{\infty}\dfrac{ 1}{2k^2\pi^2}+\dfrac{1}{2}\Vert f_{x}(x,t)\Vert_{L^2(0,1)}^2.

\end{array}$$
Therefore,  the expression of $u_{xx}(x,t)$ is uniformly convergent.
Finally,  the uniform convergence of $\,^{ABC}\,_{0}D^{\alpha}_{t}u(x,t)$, follows from equation $(\ref{sinverse})$. \\
The uniqueness of solution can be shown using the completeness properties of the system $\{\sin(k \pi x)\}.$\\
The main result for the direct problem can be summarized in the following theorem:
\begin{theorem}
Assume $f(x,t)\in C[0,1]\times C[0,T]$ such that
 $f(x,0)=0,$ $f(0,t)=f(1,t)=0,$ $f_{t}(x,t)\in L[0,T]$ and $f_{x}(x,t)\in L^{2}[0,1],$
then the problem (\ref{sinverse}) - (\ref{ISIC}) has  a unique solution $ u(x,t) $  given by
$$\begin{array}{ll}
u(x,t)&= 
\displaystyle\sum_{k=1}^{\infty}\left( \dfrac{1-\alpha}{B(\alpha)+k^2\pi^2(1-\alpha)}f_k(t)+\dfrac{\alpha\, B(\alpha)}{(B(\alpha)+k^2\pi^2(1-\alpha))^2}\right. \\&\left.  \displaystyle\int_{0}^{t} f_k(\xi) (t-\xi)^{\alpha-1}E_{\alpha,\alpha}\left[ \dfrac{-\alpha k^2\pi^2}{B(\alpha)+k^2\pi^2(1-\alpha)}(t-\xi)^{\alpha}\right]d\xi\right) \sin(k \pi x).\end{array}$$

where,
$$
f_{k}(t)=2\displaystyle\int_{0}^{1}f(x,t) \sin(k\pi x) \,dx.$$
\end{theorem}

\begin{description}
\item[Acknowledgements.]
The first two authors acknowledge financial support from The Research Council (TRC), Oman. This work is funded by TRC under the research agreement no. ORG/SQU/CBS/13/030.
\end{description}


\begin{thebibliography}{99}

\bibitem{alkahtani} B.S.T. Alkahtani, Chua’s circuit model with Atangana–Baleanu derivative with fractional order, Chaos, Solitons and Fractals, 89(2016), pp.547-551.

\bibitem{alkatan} B.S.T. Alkahtani and A. Atangana, Controlling the wave movement on the surface of shallow water with the Caputo-Fabrizio derivative with fractional order, Chaos, Solitons and Fractals, 89(2016), pp.539-546.

\bibitem{algahtani}  R.T. Algahtani, Atangana-Baleanu derivative with fractional order applied to the model of groundwater within an unconfined aquifer, J. Nonlinear Sci. App., 9(2016), pp.3647-3654.
 

\bibitem{alserk} N. Al-Salti, E. Karimov and S. Kerbal, Boundary-value problem for fractional heat equation involving Caputo-Fabrizio derivative, New Trend in Mathematical Sciences, 4 (2016), pp. 79-89.

\bibitem{alsers} N. Al-Salti, E. Karimov and K. Sadarangani , On a Differential Equation with Caputo-Fabrizio Fractional Derivative of Order $1< \beta\leq 2$ and Application to Mass-Spring-Damper System, Progr. Fract. Differ. Appl., 4 (2016), pp. 257-263.


\bibitem{atang} A. Atangana and D. Baleanu, New Fractional Derivatives with Nonlocal and Non-Singular Kernel: Theory and Application to Heat Transfer Model, Thermal Science, 20(2016), pp.763-769.

\bibitem{atangk} A. Atangana and I. Koca, Chaos in a simple nonlinear system with Atangana-Baleanu derivatives with fractional order, Chaos, Solitons and Fractals, 89(2016), pp.446-454.

\bibitem{CF} M. Caputo and and M. Fabrizio, A new Definition of Fractional Derivative without Singular Kernel, Progr. Fract. Differ. Appl. 1, 2 (2015), pp. 73-85.

\bibitem{CF2}  M. Caputo and and M. Fabrizio, Applications of New Time and Spatial Fractional Derivatives with Exponential Kernels, Progr. Fract. Differ. Appl. 2, 1 (2016), pp. 1-11.


\bibitem{jefainalgahtani} O. Jefain and  J. Algahtani, Comparing the  Atangana -Baleanu and Caputo-Fabrizio derivative with fractional order: Allen Cahn model, Chaos, Solitons and Fractals, 89(2016), pp.552-559.

\bibitem{LN} J. Losada and J. Nieto , Properties of a New Fractional Derivative without
Singular Kernel, Progr. Fract. Differ. Appl., 2 (2015), pp. 87-92.

\bibitem{pod} I. Podlubny,  Fractional differential equations. Academic Press Inc., San Diego, CA, 1999.
\bibitem{Prabhakar} T. R. Prabhakar, A singular integral equation with a generalized Mittag-Leffler function in the kernal, Yokohama Math. J., 19 (1971), pp. 7-15.

\bibitem{sheikh} N.A. Sheikh, F. Ali, M. Saqib, I. Khan, S.A.A. Jan, A.S. Alshomrani, M.S. Alghamdi, Comparison and analysis of the Atangana–Baleanu and Caputo–Fabrizio
fractional derivatives for generalized Casson fluid model with heat generation and chemical reaction, Results in Physics, 7 (2017), pp. 789-800.

\end{thebibliography}
\end{document}